\documentclass[12pt,a4paper,leqno]{amsart}

\usepackage{amsmath, amsthm}
 \input amssymb.sty

\def\deg{\operatorname{deg}}

\newtheorem{thm}{Theorem}
\newtheorem{theorem}{Theorem}
\newtheorem*{theorem*}{Theorem}
\newtheorem{lemma}[thm]{Lemma}
\newtheorem*{thm*}{Theorem}
\newtheorem*{thmA*}{Theorem A}
\newtheorem*{thmB*}{Theorem B}
\newtheorem*{TheoremCFZ1*}{Theorem CFZ1}
\newtheorem*{TheoremCFZ2*}{Theorem CFZ2}
\newtheorem*{TheoremCFZ3*}{Theorem CFZ3}
\newtheorem*{theoremB*}{Theorem B}
\newtheorem*{theoremA*}{Theorem A}
\newtheorem*{theoremPL2*}{Theorem PL2}
\newtheorem*{TheoremPZ*}{Theorem PZ}
\newtheorem*{TheoremLN*}{Theorem LN}
\newtheorem*{theoremLu*}{Theorem Lu}

\newtheorem*{cor*} {Corollary}

\newtheorem {claim} {Claim}

 \newtheorem*{defn*} {Definition}

\newtheorem*{ZL*}{Zalcman's Lemma}
\newtheorem*{Zp*}{ZP1 Lemma}
\newtheorem*{ZP*}{ZP2 Lemma}
\newtheorem*{Ne*}{N Lemma}
\newtheorem*{claim*}{Claim}

  \theoremstyle{definition}

\newtheorem{examp}{Example}
\newtheorem*{examp*}{Example}
\newtheorem*{remark*}{Remark}

 \theoremstyle{remark}

\newtheorem*{notation*}{Notation}

\newcommand{\lemref}[1]{Lemma~\ref{#1}}

\newcommand{\C}{\mathbb{C}}
\newcommand{\D}{\displaystyle}
\newcommand{\DF}[2]{\frac{\D#1}{\D#2}}

 \renewcommand{\sectionmark}[1]{}

 \newcommand{\CF}{\mathcal F}

\newcommand{\ve}{\varepsilon}

\begin{document}

\baselineskip20pt

 \title[A criterion of normality based on a single holomorphic function II]
 {A criterion of normality based on a single holomorphic function II}
\author[Xiaojun Liu and Shahar Nevo]{Xiaojun Liu$^{1}$ and Shahar Nevo$^{2}$}
\thanks{$^1$ Research supported by the NNSF of China Approved
No.11071074 and also supported by the Outstanding Youth Foundation of Shanghai No. slg10015.}

\thanks{$^2$ Research supported by the Israel Science Foundation Grant No. 395/07}

\address{Xiaojun Liu, Department of Mathematics\\
University of Shanghai for Science and Technology, Shanghai 200093,
P.R. China} \email{Xiaojunliu2007@hotmail.com}

\address{Shahar Nevo, Department of Mathematics\\
Bar-Ilan University, 52900 Ramat-Gan, Israel}
\email{nevosh@macs.biu.ac.il}


\begin{abstract}
In this paper, we continue to discuss   normality based on a
single\linebreak holomorphic function. We obtain the following
result. Let $\CF$ be a family of functions holomorphic on a domain
$D\subset\mathbb C$. Let $k\ge2$ be an integer and let
$h(\not\equiv0)$ be a holomorphic function on $D$, such that
$h(z)$ has no common zeros with any $f\in\CF$. Assume also that
the following two conditions hold for every $f\in\CF$:\linebreak
(a) $f(z)=0\Longrightarrow f'(z)=h(z)$ and 
(b) $f'(z)=h(z)\Longrightarrow|f^{(k)}(z)|\le c$, where $c$ is a
constant. Then $\CF$ is normal on $D$.

A geometrical approach is used to arrive at the result which
significantly improves the previous results of the authors,
\textit{A criterion of normality based on a single holomorphic
function}, Acta Math. Sinica, English Series (1) \textbf{27}
(2011), 141--154 and of Chang, Fang, and Zalcman, \textit{Normal
families of holomorphic functions}, Illinois Math. J. (1)
\textbf{48} (2004), 319--337.
 We~also deal with two other similar criterions of
normality. Our results are shown to be sharp.
\end{abstract}


\keywords{Normal family, holomorphic functions, zero points}

\subjclass[2010]{30D35}

 \maketitle

 \section{Introduction}\label{introduction}

In \cite{11}, X.C. Pang and L. Zalcman proved the following theorem.

\begin{TheoremPZ*} Let $\CF$ be a family of meromorphic
functions on a domain  $D\subset\C$, all of whose zeros have
multiplicity at least $k$, where $k\ge1$ is an integer. Suppose
there exist constants $b\ne0$ and $h>0$ such that, for every
$f\in\CF$, $f(z)=0\Longleftrightarrow f^{(k)}(z)=b$ and
$f(z)=0\Longrightarrow0<|f^{(k+1)}(z)|\le h$. Then $\CF$ is a normal
family on $D$.
\end{TheoremPZ*}

Then, in \cite{1}, J.M Chang, M.L. Fang, and L. Zalcman proved the
following result.

\begin{TheoremCFZ1*}\cite[Theorem 4]{1}
Let $\CF$ be a family of functions holomorphic on a domain
$D\subset\C$. Let $k\ge2$ be an integer, and let $h(z)\ne0$ be a
function analytic in $D$. Assume also that the following two
conditions hold for every $f\in\CF$:
\begin{enumerate} \item[(a)] $f(z)=0\Longrightarrow f'(z)=h(z)$;
and\item[(b)] $f'(z)=h(z)\Longrightarrow|f^{(k)}(z)|\le c$, where
$c$ is a constant.
\end{enumerate}
Then $\CF$ is normal on $D$.
\end{TheoremCFZ1*}

And in \cite{4}, we replaced the condition $h(z)\ne0$ with
$h(z)\not\equiv0$ and obtained the following result.

\begin{TheoremLN*}
Let $\CF$ be a family of functions holomorphic on a domain
$D\subset\C$. Let $k\ge2$ be an integer, and let
$h(z)(\not\equiv0)$ be a holomorphic function on $D$, all of whose
zeros have multiplicity at most $k-1$,  that has no common zeros
with any $f\in\CF$. Assume also that the following two conditions
hold for every $f\in\CF$:
\begin{enumerate} \item[(a)] $f(z)=0\Longrightarrow f'(z)=h(z)$
and\item[(b)] $f'(z)=h(z)\Longrightarrow|f^{(k)}(z)|\le c$, where
$c$ is a constant.
\end{enumerate}
Then $\CF$ is normal on $D$.
\end{TheoremLN*}

We now pose the following question: can the restriction for the
zeros of $h(z)$ with multiplicity at most $k-1$ be dropped? In
this paper, we continue to study the above problem and obtain an
affirmative answer.

\begin{theorem}\label{thm1} Let $\CF$ be a family of functions holomorphic on a domain
$D\subset\C$. Let $k\ge2$ be an integer, and let
$h(z)(\not\equiv0)$ be a holomorphic function on $D$   that  has
no common zeros with any $f\in\CF$. Assume also that the following
two conditions hold for every $f\in\CF$:
\begin{enumerate} \item[(a)] $f(z)=0\Longrightarrow f'(z)=h(z)$
and\item[(b)] $f'(z)=h(z)\Longrightarrow|f^{(k)}(z)|\le c$, where
$c$ is a constant.
\end{enumerate}
Then $\CF$ is normal on $D$.
\end{theorem}

Also in \cite{1},  the case for the $k-$th derivative was
considered  and the following result was proved .

\begin{TheoremCFZ2*}\cite[Theorem 1]{1}
Let $\CF$ be a family of functions holomorphic on a domain
$D\subset\C$, all of whose zeros have multiplicity at least $k$,
where $k\ne2$ is a positive integer; and let $h(z)\ne0$ be a
function analytic in $D$. Assume also that the following two
conditions hold for every $f\in\CF$:
\begin{enumerate} \item[(a)] $f(z)=0\Longrightarrow f^{(k)}(z)=h(z)$;
and\item[(b)] $f^{(k)}(z)=h(z)\Longrightarrow|f^{(k+1)}(z)|\le c$,
where $c$ is a constant.
\end{enumerate}
Then $\CF$ is normal on $D$.
\end{TheoremCFZ2*}

For the case $k=2$,   the following result was obtained.

\begin{TheoremCFZ3*}\cite[Theorem 3]{1}
Let $\CF$ be a family of functions holomorphic on a domain
$D\subset\C$, all of whose zeros are multiple, where $s\ge4$ is an
even integer; and let $h(z)\ne0$ be a function analytic in $D$.
Assume also that the following two conditions hold for every
$f\in\CF$:
\begin{enumerate} \item[(a)] $f(z)=0\Longrightarrow f''(z)=h(z)$;
and\item[(b)] $f''(z)=h(z)\Longrightarrow|f'''(z)|+|f^{(s)}(z)|\le
c$, where $c$ is a constant.
\end{enumerate}
Then $\CF$ is normal on $D$.
\end{TheoremCFZ3*}

In view of the improvement of Theorems CFZ1 and LN via Theorem 1,
the question that naturally arises concerning Theorem CFZ2 and
CFZ3, is whether the condition $h(z)\ne0$, $z\in D$, can be
relaxed to ``$h \not\equiv0$ ". It turns out that the answer is
negative in both cases. It is negative even if $h$ has no common
zero with any $f\in\mathcal F$ (like  in Theorem 1). To construct
the first example, concerning Theorem CFZ2, we first need to
mention the following famous result of F. Lucas.

\begin{theoremLu*} \cite{5}, \cite[p.~22]{6}
Let $P(z)$ be a nonconstant polynomial. Then all the zeros of
$P'(z)$ lie in the convex hull $H$ of the zeros of $P(z)$.
Moreover, there are no zeros of $P'(z)$ on the boundary of $H$,
unless this zero is a multiple zero of $P(z)$ or   the zeros of
$P(z)$ are colinear.
\end{theoremLu*}

\begin{examp}
Let $r\ge1$ and $k\ge3$ be integers, $D=\Delta$ be the unit disc
and $h(z)=z^r$. Define
$$f_n(z)=a_n\left(z^\ell-\DF1{n^\ell}\right)^k,$$
where $\ell=k+r$ and $a_n=\DF{n^{(k-1)\ell}}{k!\ell^k}$.

We have$$f_n(z)=a_n\prod\limits_{j=1}^\ell\left(z-\alpha^{(n)}_j
\right)^k,$$where $\alpha^{(n)}_j=\DF{\exp\left(i\frac{2\pi j}\ell
\right)}n$, for $1\le j\le\ell$.

By calculation,
\begin{align*}f_n^{(k)}\left(\alpha^{(n)}_j\right)&=k!a_n
\prod\limits_{t=1,t\ne
j}^\ell\left(\alpha^{(n)}_j-\alpha^{(n)}_t\right)^k=k!a_n\left[
\left(z^\ell-\DF1{n^\ell}\right)'\Bigg|_{z=\alpha^{(n)}_j}\right]^k\\
&= k!a_n\ell^k\left(\alpha^{(n)}_j\right)^{k(\ell-1)}.\end{align*}
Thus,\begin{equation}\label{1}\arg\left[f_n^{(k)}\left(\alpha^{(n)}_j\right)\right]=
(\ell-1)k\cdot\DF{2\pi j}\ell=-\DF{2\pi kj}\ell=\DF{2\pi
ri}{\ell}=\arg\left[z^r\Big|_{z=\alpha^{(n)}_j}\right].
\end{equation}
Here the equalities are modulo $2\pi$, and we used in the last equality that $r+k=\ell$.

We have
\begin{equation}\label{2}\left|f_n^{(k)}\left(\alpha^{(n)}_j\right)\right|=\DF{k!\ell^kn^{\ell(k-1)}}
{k!\ell^k}\left(\DF1n\right)^{k(\ell-1)}=\left(\DF1n\right)^r=\left|z^r\right|
\Bigg|_{z=\alpha^{(n)}_j}.
\end{equation}
From \eqref{1} and \eqref{2} we have that $f_n(z)=0\Longrightarrow
f^{(k)}_n(z)=h(z)$, i.e., assumption (a) of Theorem CFZ2 holds.

In order to confirm (b) of Theorem CFZ2, set
$$\widetilde{f}_n(z)=f_n(z)-\DF{z^\ell}{\ell(\ell-1)\cdots(r+1)}.$$
We have $f^{(k)}_n(z)=h(z)\Longleftrightarrow\widetilde{f}^{(k)}_n(z)=0$.

Now
\begin{equation}
\label{3}\widetilde{f}_n(z)=0
\Longleftrightarrow\DF{n^{k(\ell-1)-r}}{k!\ell^k}\left(z^\ell-\DF1{n^\ell}\right)^k
=\DF{z^\ell}{\ell(\ell-1)\cdots(r+1)}.
\end{equation}

Suppose by negation that there exist a sequence
$\{z_n\}^\infty_{n=1}$ $(z_n\to0)$ and a sequence of~natural
numbers $\{k_n\}^\infty_{n=1}$
$(k_n\underset{n\to\infty}\longrightarrow\infty)$, such that
$\widetilde{f}_{k_n}(z_n)=0$. Then since\newline
$\DF{(k_nz_n)^\ell-1}{(k_nz_n)^\ell}\underset{n\to\infty}\longrightarrow1$,
from \eqref{3} we get
\begin{equation}
\label{4}\DF{k_n^{(k-1)\ell}(k_nz_n)^{k\ell}}{k_n^{k\ell}z^\ell_n}\underset{n\to\infty}\longrightarrow
\DF{k!\ell^k}{\ell(\ell-1)\cdots(r+1)}.
\end{equation}
But the left hand side of \eqref{4} tends to $\infty$, as
$n\to\infty$, a contradiction.

We deduce that there exists some $0<C_1<\infty$, such that every
zero $z_n$ of $\widetilde{f}_n$ satisfies $|z_n|\le\DF{C_1}n$. By
Theorem Lu, we have also $|\widehat{z}_n|\le\DF{C_1}n$ for every
$\widehat{z}_n$, which is a zero of $\widetilde{f}^{(k)}_n$. But
those $\{\widehat{z}_n\}$ are exactly the points where
$f^{(k)}_n(z)=h(z)$.

Hence $f^{(k)}_n(z)=h(z)$ implies that $|z|\le\DF{C_1}n$, and we
have only to prove the following claim.

\begin{claim}\label{claim1} There exists $0<C<\infty$, such that
$|z|\le\DF{C_1}n$ implies $|f^{(k+1)}_n(z)|\le C$.\end{claim}

\begin{proof} We have
$f_n(z)=\DF{n^{(k-1)\ell}}{k!\ell^k}\left(z^\ell-\DF1{n^\ell}\right)^k=
\DF{n^{(k-1)\ell}}{k!\ell^k}\sum\limits_{j=0}^k\binom{k}{j}z^{\ell
j}\left(\DF1n\right)^{\ell(k-j)}(-1)^{k-j}$. Thus, since $\ell
j\ge k+1$ only for $j\ge1$, we get
that$$f^{(k+1)}_n(z)=\DF{n^{(k-1)\ell}}{k!\ell^k}\sum\limits_{j=1}^k\binom{k}{j}
\left(\DF1n\right)^{\ell k-\ell j}(-1)^{k-j}\ell j(\ell
j-1)\cdots(\ell j-k-1) z^{\ell j-k-1}.$$Thus, if
$|z|\le\DF{C_1}n$, then
$$\begin{aligned}|f^{(k+1)}_n(z)|&\le\DF{n^{(k-1)\ell}}{k!\ell^k}\sum\limits_{j=1}^k\binom{k}{j}
C^{\ell j-k-1}_1\ell j(\ell j-1)\cdots(\ell j-k-1)n^{k+1-\ell j}\cdot n^{\ell j-\ell k}\\
\\&=\DF{n^{k+1-\ell}}{k!\ell^k}\sum\limits_{j=1}^k\binom{k}{j}
C^{\ell j-k-1}_1\ell j(\ell j-1)\cdots(\ell j-k-1)\le
C,\end{aligned}
$$where $C=\DF1{k!\ell^k}\sum\limits_{j=1}^k\binom{k}{j}
C^{\ell j-k-1}_1\ell j(\ell j-1)\cdots(\ell j-k-1)$. (Here we used
that $k+1-\ell\le0$.)    The Claim  is proved.
\end{proof}

Hence, $\{f_n\}$ with $h$ satisfy (a) and (b) of Theorem CFZ2, but $\{f_n\}$ is not
normal at $z=0$.

Observe that when $k=1$, then $a_n=\DF1{\ell}\not\to\infty$, and
we do not get a non-normal family, as expected by Theorem 1.
\end{examp}
The following example shows that the condition $h(z)\ne0$ is
essential also for\linebreak Theorem~CFZ3.
\begin{examp} (cf.~\cite[Ex.~4]{1}
Let $s\ge4$ be an even integer and consider the family $\mathcal F=\{f_n(z)\}^\infty_{n=1}$,
$$f_n(z)=\DF{n^s}{2s^2}\left(z^s-\DF1{n^s}\right)^2\quad\text{on}\quad
\Delta.$$Let $h(z)=z^{s-2}$.

We have that
$$f_n(z)=\DF{n^s}{2s^2}\prod\limits_{j=1}^s\left(z-\alpha
^{(n)}_j\right)^2,$$ where $\alpha^{(n)}_j=\DF{\exp(i2\pi j/s)}n$,
$1\le j\le s$.

By calculation we have
\begin{equation}
\label{5}f''_n(z)=\DF{n^s}{s}
\left((2s-1)z^{s}-\DF{(s-1)}{n^s}\right)z^{s-2},
\end{equation}
\begin{align}
f'''_n(z)&=\DF{n^s}{s}\left[(2s-1)(2s-2)z^{s}-\DF{(s-1)(s-2)}{n^s}\right]z^{s-3}\label{6}\\
&=\frac{n^s}{s}(s-1)z^{s-3}\left[(4s-2)z^s-\frac{s-2}{n^s}\right],\nonumber
\end{align}
and
\begin{equation}
\label{7}f^{(s)}_n(z)=\DF{n^s}{s}\left[(2s-1)(2s-2)\cdots(s+1)z^{s}-\DF{(s-1)!}{n^s}\right].
\end{equation}
Now, if $f_n(z)=0$, then $z=\alpha^{(n)}_j$ for some $1\le j\le s$, and
thus $z^s=\DF1{n^s}$ and by \eqref{5}, $f''_n(z)=z^{s-2}=h(z)$.

If $f''_n(z)=z^{s-2}=h(z)$, then by \eqref{5}, $z=0$ or $z=\alpha^{(n)}_j$,
$1\le j\le s$. By \eqref{6} and \eqref{7}, we get
\begin{equation}
\label{8}f^{(3)}_n(0)=0,\quad f^{(s)}_n(0)=-\DF{(s-1)!}{n^s}
\end{equation}
and
\begin{equation}
\label{9}f^{(3)}_n\left(\alpha^{(n)}_j\right)=3(s-1)\DF1{n^{s-3}},\quad
f^{(s)}_n\left(\alpha^{(n)}_j\right)=\DF1s\left[\DF{(2s-1)!}{s!}-(s-1)!\right].
\end{equation}
From \eqref{8} and \eqref{9}, we see that the family $\mathcal F$ with
$h$ satisfy assumption (a) and (b) of Theorem CFZ3, but $\mathcal F$ is not normal
at $z=0$. Indeed, the reason must be that $h(0)=0$.
\end{examp}

In Example 1, we have that $f^{(k+1)}(z)\ne0$ at the zero points
of $f^{(k)}(z)-h(z)$. If we strengthen  condition (b) of Theorem
CFZ2 to be $f^{(k)}(z)=h(z)\Longrightarrow f^{(k+1)}(z)=0$, then
we can obtain the following normal criterion.

\begin{theorem}\label{thm2} Let $\CF$ be a family of functions holomorphic on a domain
$D\subset\C$, all of whose zeros have multiplicity at least $k$,
where $k\ne2$ be a positive integer. Let $h(z)(\not\equiv0)$ be a
holomorphic function on $D$, that has no common zeros with any
$f\in\CF$. Assume also that the following two conditions hold for
every $f\in\CF$:
\begin{enumerate} \item[(a)] $f(z)=0\Longrightarrow
f^{(k)}(z)=h(z)$; and\item[(b)] $f^{(k)}(z)=h(z)\Longrightarrow
f^{(k+1)}(z)=0$.
\end{enumerate}
Then $\CF$ is normal on $D$.
\end{theorem}

Similarly, if we strengthen the condition (b) of Theorem CFZ3 to
$f''(z)=h(z)\Longrightarrow f'''(z)=f^{(s)}(z)=0$, then we can
also obtain the normality criterion.

\begin{theorem}\label{thm3} Let $\CF$ be a family of functions holomorphic on a domain
$D\subset\C$, all of whose zeros are multiple, where $s\ge2$ is an
even integer. Let $h(z)(\not\equiv0)$ be a holomorphic function on
$D$, that has no common zeros with any $f\in\CF$. Assume also that
the following two conditions hold for every $f\in\CF$:
\begin{enumerate} \item[(a)] $f(z)=0\Longrightarrow
f''(z)=h(z)$; and\item[(b)] $f''(z)=h(z)\Longrightarrow
f'''(z)=f^{(s)}(z)=0$.
\end{enumerate}
Then $\CF$ is normal on $D$.
\end{theorem}

Before we go to the proofs of the main results, let us set some
notation. Throughout, $D$ is a domain in $\C$. For $z_0\in\C$ and
$r>0$, $\Delta(z_0,r)=\{z:|z-z_0|<r\}$  and
$\Delta'(z_0,r)=\{z:0<|z-z_0|<r\}$. The unit disc will be denoted
by $\Delta$ and $\C^\ast=\C\setminus\{0\}$. We write
$f_n(z)\overset\chi \Rightarrow f(z)$ on $D$ to indicate that the
sequence $\{f_n\}$ converges to $f$  in the spherical metric,
uniformly on compact subsets of $D$, and $f_n\Rightarrow f$ on $D$
if the convergence is in the Euclidean metric. For a meromorphic
function $f(z)$ in $D$ and $a\in\widehat{\C}$,
$\overline{E}_f(a):=\{z\in D:f(z)=a\}$. The spherical derivative
of the meromorphic function $f$ at the point $z$ is denoted by
$f^\#(z).$

Frequently, given a sequence $\{f_n\}_1^\infty$ of functions, we
need to extract an appropriate subsequence; and this necessity may
recur within a single proof. To avoid the awkwardness of multiple
indices, we again denote the extracted subsequence by $\{f_n\}$
(rather than, say, $\{f_{n_k}\})$ and designate this operation by
writing ``taking a subsequence and renumbering," or simply
``renumbering". The same convention applies to sequences of
constants.

The plan of the paper is as follows. In Section 2, we state a
number of preliminary results. Then  in Section 3  we prove
Theorem 1. Finally, in Section 4  we prove Theorem~2.
\section{Preliminary results}

The following lemma is the local version of a well-known lemma of
X. C. Pang and L. Zalcman \cite[Lemma 2]{11}. For a proof see
\cite[Lemma 2]{4}, also cf.\cite[Lemma 2]{9},
\cite[pp.~216--217]{15}, \cite[pp.~299--300]{7}, \cite[p.~4]{8}.

\begin{lemma}\label{lemma1} Let $\CF$ be a family of functions meromorphic in
a domain $D$, all of whose zeros have multiplicity at least $k$, and
suppose that there exists $A\ge1$, such that $|f^{(k)}(z)|\le A$
whenever $f(z)=0$. Then if $\CF$ is not normal at $z_0\in D$, there
exist, for each $0\le\alpha\le k$,
\begin{enumerate} \item[(a)] points $z_n\to z_0$;
\item[(b)] functions $f_n\in\CF$;and\item[(c)] positive numbers
$\rho_n\to0^+$
\end{enumerate}
such that $g_n(\zeta):=\rho^{-\alpha}_nf_n(z_n+f_n\zeta)
\overset\chi\Rightarrow g(\zeta)$ on $\C$, where $g$ is a
nonconstant meromorphic function on $\C$, such that for every
$\zeta\in\C$, $g^\#(\zeta)\le g^\#(0)=kA+1$.
\end{lemma}

\begin{lemma}\label{lemma2}\cite[Lemma 5]{1} Let $f$ be a nonconstant entire
function of order $\rho$, $ 0\le\rho\le1$, all of whose zeros have
multiplicity at least $k$, where $k\ne2$ is a positive integer.
And let $a\ne0$ be a constant. If
$\overline{E}_f(0)\subset\overline{E}_
{f^{(k)}}(a)\subset\overline{E}_{f^{(k+1)}}(0)$, then
$$f(z)=\DF{a(z-b)^k}{k!},$$ where $b$ is a constant.
\end{lemma}

\begin{lemma}\label{lemma3}\cite[Lemma 6]{1} Let $f$ be a nonconstant entire
function of order $\rho$, $ 0\le\rho\le1$, all of whose zeros are
multiple. Let $s\ge 4$ be an even integer and $a\ne0$ be a
constant. If
$\overline{E}_f(0)\subset\overline{E}_{f''}(a)\subset\overline{E}_{f'''}(0)\cap
\overline{E}_{f^{(s)}}(0)$, then$$f(z)=\DF{a(z-b)^2}2,$$where $b$
is a constant.
\end{lemma}

\begin{lemma}\label{lemma4}(see \cite[pp.~118--119,122--123]{2}) Let $f$ be a meromorphic
function on $\C$. If $f^\#$ is uniformly bounded on $\C$, then the
order of $f$ is at most $2$. If $f$ is an entire function, then the
order of $f$ is at most $1$.
\end{lemma}

The following lemma is a slight generalization of Theorem CFZ2 for
sequences.

\begin{lemma}\label{lemma5}(cf. \cite[Lemma 5]{4}) Let $\{f_n\}$ be a sequence of functions holomorphic on a domain
$D\subset\C$, all of whose zeros have multiplicity at least $k$,
and let $\{h_n\}$ be a sequence of functions analytic on $D$ such
that $h_n {(z)}\Rightarrow h {(z)}$ on $D$, where $h(z)\ne0$ for
$z\in D$ and $k\ne2$ be a positive integer. Suppose that, for each
$n$, $f_n(z)=0\Longrightarrow f^{(k)}_n(z)=h_n(z)$ and
$f^{(k)}_n(z)=h_n(z)\Longrightarrow f^{(k+1)}_n(z)=0$. Then
$\{f_n\}$ is normal on $D$.
\end{lemma}

\begin{proof} Suppose to the contrary that there exists $z_0\in D$ such that
$\{f_n\}$ is not normal at $z_0$. The convergence of $\{h_n\}$ to
$h$ implies that, in some neighborhood of $z_0$, we have
$f_n(z)=0\Rightarrow |f_n^{(k)}(z)| \le |h(z_0)|+1$ (for large
enough $n$). Thus we can apply \lemref{lemma1} with $\alpha=k$ and
$A$  such that $kA+1>\max\Big\{|h(z_0)|+1,\DF{|h(z_0)|}{(k-1)!},
\DF{k\cdot k!}{|h(z_0)|}\Big\}=\max\Big\{|h(z_0)|+1,\DF{k\cdot k!}
{|h(z_0)|}\Big\}$. So we can take an appropriate subsequence of
$\{f_n\}$ (denoted also by $\{f_n\}$ after renumbering), together
with points $z_n\to z_0$ and positive numbers $\rho_n\to0^+$ such
that$$g_n(\zeta)=\frac{f_n(z_n+\rho_n
\zeta)}{\rho^k_n}\overset\chi\Longrightarrow
g(\zeta)\quad\text{on}\quad\C,$$where $g$ is a nonconstant entire
function and $g^\sharp(\zeta)\le
g^\sharp(0)=kA+1=k(|h(z_0)|+1)+1$. We claim that
\begin{equation}
\label{10}
\overline{E}_g(0)\subset\overline{E}_{g^{(k)}}(h(z_0))
\subset\overline{E}_{g^{(k+1)}}(0).
\end{equation}

In fact, if there exists $\zeta_0\in\C$, such that $g(\zeta_0)=0$,
then since $g(\zeta)\not\equiv0$, there exist $\zeta_n$,
$\zeta_n\to\zeta_0$, such that if $n$ is sufficiently large,
$$g_n(\zeta_n)=\frac{f_n(z_n+\rho_n\zeta_n)}{\rho^k_n}=0.$$ Thus
$f_n(z_n+\rho_n\zeta_n)=0$, so that $f^{(k)}_n(z_n+\rho_n\zeta_n)
=h_n(z_n+\rho_n\zeta_n)$, i.e., that
$g^{(k)}_n(\zeta_n)=h_n(z_n+\rho_n\zeta_n)$. Since
$g^{(k)}(\zeta_0)=\lim\limits_{n\to\infty}g^{(k)}_n(\zeta_n)=h(z_0)$,
we have established the first part of the Claim  that
$\overline{E}_g(0) \subset\overline{E} _{g^{(k)}}(h(z_0))$.

Now, suppose there exists $\zeta_0\in\C$, such that
$g^{(k)}(\zeta_0)=h(z_0)$. If $g^{(k)}(\zeta)\equiv h(z_0)$, then
$g^{(k+1)}\equiv0$ and we are done. Thus we can assume that
$g^{(k)}$ is not constant and since $f^{(k)}_n(z_n+\rho_n\zeta
)-h_n(z_n+\rho_n\zeta)\Rightarrow g^{(k)} (\zeta)-h(z_0)$,  we get
by Hurwitz's Theorem that there exist $\zeta_n$,
$\zeta_n\to\zeta_0$, such that$$f^{(k)}_n(z_n+\rho_n\zeta
_n)-h_n(z_n+\rho_n\zeta_n)=g^{(k)}_n(\zeta_n)-h_n(z_n+\rho_n\zeta_n)=0.$$
Thus we have $f^{(k+1)}_n(z_n+\rho_n\zeta_n)=0$ and
$g^{(k+1)}_n(\zeta_n)=0$. Letting $n\to\infty$, we get that
$g^{(k+1)}(\zeta_0)=0$. This completes the proof of the Claim.
Now, by Lemmas \ref{lemma4} and \ref{lemma2}, we have
$g(\zeta)=\DF{h(z_0)(\zeta-\zeta_1)^k}{k!}$, where $\zeta_1$ is a
constant. Thus
$$g^\sharp(0)=\DF{|h(z_0)||\zeta_1|^{k-1}/(k-1)!}{1+|h(z_0)|^2
|\zeta_1|^{2k}/k!^2}.$$ Now, if $|\zeta_1|\le1$, then
$g^\sharp(0)\le \DF{|h(z_0)|}{(k-1)!}<kA+1$, and if $|\zeta_1|>1$,
then $g^\sharp(0)\le
\DF{|h(z_0)||\zeta_1|^{k-1}/(k-1)!}{|h(z_0)|^2|\zeta_1|^{2k}/k!^2}\le
\DF{k\cdot k!}{|h(z_0)|}<kA+1$. In either case we get a
contradiction.
\end{proof}

Similarly, we can get a slight generalization of Theorem CFZ3 for
sequences.

\begin{lemma}\label{lemma6} Let $\{f_n\}$ be a sequence of functions holomorphic on a domain
$D\subset\C$, all of whose zeros are multiple and $\{h_n\}$ be a
sequence of functions analytic on $D$ such that $h_n
{(z)}\Rightarrow h {(z)}$ on $D$, where $h(z)\ne0$ for $z\in D$ and
$s\ge2$ be an even integer. Suppose that, for each $n$,
$f_n(z)=0\Longrightarrow f''_n(z)=h_n(z)$ and
$f''_n(z)=h_n(z)\Longrightarrow f'''(z)=f^{(s)}_n(z)=0$, then
$\{f_n\}$ is normal on $D$.
\end{lemma}

The proof is very similar to the proof of Lemma \ref{lemma5}. We
start to argue the same (with $2$ instead of $k$), and then
instead of proving \eqref{10} we prove that
$$\overline{E}_g(0)\subset\overline{E}_{g''}(h(z_0))\subset\overline{E}_{g^{(3)}}(0)
\cap\overline{E}_{g^{(s)}}(0).$$The left inclusion is proved in
the same manner. Concerning the right inclusion, we now deduce
from $f''_n(z_n+\rho_n\zeta_n)- h_n(z_n+\rho_n\zeta_n)=0$ that
$f^{(3)}_n(z_n+\rho_n\zeta_n)= f^{(s)}_n(z_n+\rho_n\zeta_n)=0$.
Then, since $\rho_nf^{(3)}_n(z_n+\rho_n\zeta)\Rightarrow
g^{(3)}(\zeta)$ in $\C$ and
$\rho^{s-2}_nf^{(s)}_n(z_n+\rho_n\zeta)\Rightarrow g^{(s)}(\zeta)$
in $\C$, we conclude that $g^{(3)}(\zeta_0)=g^{(s)}(\zeta_0)=0$.
To get the final contradiction, we apply now Lemmas \ref{lemma4}
and \ref{lemma3} instead of Lemmas \ref{lemma4} and \ref{lemma2}.

The following result will play an essential role in treating
transcendental functions which is used in the proofs of Theorems 2
and 3.

\begin{theoremB*} $($\cite{13} see also \cite[p.~117]{2}$)$
Let $f(z)$ be a function homomorphic in\linebreak
$\{z:R<|z|<\infty\}$, with essential singularity at $z=\infty$.
Then $\varlimsup\limits_{|z|\to\infty}|z|f^\#(z)=+\infty$.
\end{theoremB*}

For the proof of Theorem 2, we need also the following Lemma.

\begin{lemma}\label{lemma7} Let $h $ be a holomorphic function on $D,$ with a
zero of order $\ell(\ge1)$ at $z_0\in D.$ Let $\{f_n\}^\infty_{n=1}$ be a
sequence of functions with zeros of multiplicity at least $k$, such that
$\{f_n\}$ and $h$ satisfy conditions {\rm(a)} and {\rm(b)} of Theorem 2. Let
$\{\alpha_n\}^\infty_{n=1}$ be a sequence of nonzero numbers such
that $\alpha_n\to0$ as $n\to\infty$. Then $\{f_n(z_0+\alpha_n\zeta)/\alpha^{k+\ell}_n\}^
\infty_{n=1}$ is normal in $\C^\ast$.
\end{lemma}
\begin{proof}
Without loss of generality, we may assume that $z_0=0$. In a
neighborhood of the origin we have $h(z)=z^{\ell}b(z)$, where
$b(z)$ is analytic, $b(0)\ne0$. Define
$r_n(\zeta)=\zeta^{\ell}b(\alpha_n\zeta)$. We will show that the
assumptions of \lemref{lemma5} hold in $\C^\ast$ for the sequence
$\{G_n(\zeta)\}^\infty_{n=1}$,
$G_n(\zeta):=f_n(\alpha_n\zeta)/\alpha^{k+\ell}_n$ and
$\{r_n(\zeta)\}^\infty_{n=1}$. First, we have that
$r_n(\zeta)\Rightarrow b(0)\zeta^{\ell}$ on $\C$ and
$\zeta^{\ell}\ne0$ in $\C^\ast$. Assume that $G_n(\zeta)=0$. Then
$f_n(\alpha_n\zeta)=0$ and
$f^{(k)}_n(\alpha_n\zeta)=(\alpha_n\zeta)^{\ell}b(\alpha_n\zeta)$,
and we get that $G^{(k)}_n(\zeta)=r_n(\zeta)$. Suppose now that
$G^{(k)}_n(\zeta)=r_n(\zeta)$. This means that
$f^{(k)}_n(\alpha_n\zeta)=h(\alpha_n\zeta)$ and thus
$f^{(k+1)}_n(\alpha_n\zeta)=0$. We have $G^{(k+1)}_n(\zeta)=0$,
and thus the assumptions of \lemref{lemma5} hold. Hence we deduce
that $\{G_n(\zeta)\}$ is normal in $\C^\ast$, and the lemma is
proved.
\end{proof}
The following lemma plays a similar role in the proof of Theorem
3, to the role of \lemref{lemma7} in the proof of Theorem 2.
\begin{lemma}\label{lemma8} Let $h $ be a holomorphic function on $D,$ with a
zero of order $\ell(\ge1)$ at $z_0\in D.$ Let
$\{f_n\}^\infty_{n=1}$ be a sequence of functions whose zeros are
multiple, such that $\{f_n\}$ and $h$ satisfy conditions {\rm(a)}
and {\rm(b)} of Theorem 3. Let $\{\alpha_n\}^\infty_{n=1}$ be a
sequence of nonzero numbers such that $\alpha_n\to0$ as
$n\to\infty$. Then
$\{f_n(z_0+\alpha_n\zeta)/\alpha^{2+\ell}_n\}^\infty_{n=1}$ is
normal in $\C^\ast$.
\end{lemma}
The proof of this lemma is analogous  to the proof of
\lemref{lemma7}. Of course, we use  \lemref{lemma6} instead of
 \lemref{lemma5}.
\section{Proof of Theorem 1}

In this section, we do not use any of the preliminary results.
The proof is elementary.

By Theorem CFZ1, $\CF$ is normal at every point $z_0\in D$ at
which $h(z_0)\ne0$(so immediately we get that $\CF$ is
quasinormal). So let $z_0$ be a zero of $h$ of order $\ell(\ge1)$.
Without loss of generality, we can assume that $z_0=0$, and then
$h(z)=z^{\ell}b(z)$. Here $b$ is an analytic function in
$\Delta(0,\delta)$ and $b(z)\ne0$ there. We assume that
$0<\delta<1$, and by taking a subsequence and renumbering, we can
assume that
\begin{equation}
\label{11}f_n\Longrightarrow
f\quad\text{in}\quad\Delta'(0,\delta).
\end{equation}
Now, if $f$ is holomorphic in $\Delta'(0,\delta)$, we deduce by
the maximum principle that $f_n\Rightarrow f$ on
$\Delta(0,\delta)$, and we are done. So let us assume that
$f_n\Rightarrow\infty$ in $\Delta'(0,\delta)$. Fix $\eta$,
$0<\eta<\delta$. By the minimum principle (i.e., the maximum
principle for $\{1/f_n\}$), there exists $N=N(\eta)$, such that
for every $n\ge N$, $f_n$ has $k_n(k_n\ge1)$ simple zeros in
$\overline{\Delta}(0,\eta)-\{0\}$, say $\alpha^{(n)}_1$,
$\alpha^{(n)}_2$, $\cdots$, $\alpha^{(n)}_{k_n}$ (otherwise we get
that $f_n\Rightarrow\infty$ in $\Delta(0,\eta)$ and we are done).
Since $f_n\Rightarrow\infty$ in $\Delta'(0,\delta)$, we get that
\begin{equation}
\label{12}\max\limits_{1\le j\le k_n}\{|\alpha^{(n)}_j|\}\to0,\quad\text{as}\quad n\to\infty.
\end{equation}
We can write
$f_n(z)=t_n(z)\prod\limits_{i=1}^{k_n}\left(z-\alpha^{(n)}_i\right)$,
where $t_n(z)\ne0$ for $z\in\overline{\Delta}(0,\eta)$ and $n\ge
N$. Since $\eta<1$, we get by \eqref{12} that
$\DF{t_n(z)}{b(z)}\Rightarrow\infty$ in
$\overline{\Delta}(0,\eta)$. By condition (a) of Theorem~1, we
have, for $n\ge N$,
$f'_n(\alpha^{(n)}_j)=\alpha^{(n)\ell}_jb(\alpha^{(n)}_j)$, $1\le
j\le k_n$. By calculation,
$$f'_n(z)=t'_n(z)\prod\limits_{i=1}^{k_n}\left(z-\alpha^{(n)}_i\right)+t_n(z)\left[
\prod\limits_{i=1}^{k_n}\left(z-\alpha^{(n)}_i\right)\right]',$$and so
\begin{equation}
\label{13} t_n\left(\alpha^{(n)}_j\right)\left[
\prod\limits_{i=1}^{k_n}\left(z-\alpha^{(n)}_i\right)\right]'\Bigg|_{z=\alpha^{(n)}_j}=
\alpha^{(n)\ell}_jb\left(\alpha^{(n)}_j\right).
\end{equation}
Define, for $n\ge N$,
$$M_n(z):=\DF{t_n(z)}{b(z)}\left[
\prod\limits_{i=1}^{k_n}\left(z-\alpha^{(n)}_i\right)\right]'-z^{\ell}.$$
By \eqref{13} we get that $M_n\left(\alpha^{(n)}_j\right)=0$ for
$1\le j\le k_n$, and so for $n\ge N,$ $M_n$ has at least $k_n$
zeros in $\Delta'(0,\eta)$, including multiplicities. Here we use
the fact $h$ has no common zero with any $f_n.$ Since such a zero
must be $z=0$ and would be a zero of order $m$ (must be $m\ge2$ by
condition (a)) of $f_n$, and it would be a zero of order $m-1$ of
$M_n$ (if $\ell>m-1$) or even of order $\ell<m-1$ (if $\ell<m-1$),
then we would not know that the number of zeros (including
multiplicities) of $M_n$ is at least $k_n$. This fact, under the
assumption that there are no common zeros, will lead to the
desired contradiction.

\begin{claim}\label{claim2} $\DF{t_n(z)}{b(z)}\left[
\prod\limits_{i=1}^{k_n}\left(z-\alpha^{(n)}_i\right)\right]'\Rightarrow\infty$\quad
in\quad $\Delta'(0,\eta)$.
\end{claim}

\begin{proof}
We write
\begin{equation}
\label{14} \DF{t_n(z)}{b(z)}\left[
\prod\limits_{i=1}^{k_n}\left(z-\alpha^{(n)}_i\right)\right]'=\sum\limits_{j=1}^{k_n}\DF{t_n(z)}{b(z)}
\prod\limits_{i=1,i\ne j}^{k_n}\left(z-\alpha^{(n)}_i\right).
\end{equation}
For any $\ve$, $0<\ve<\eta$, we have that
\begin{equation}
\label{15}\DF{t_n(z)}{b(z)}\prod\limits_{i=2}^{k_n}\left(z-\alpha^{(n)}_i\right)\Longrightarrow\infty
\quad\text{in}\quad\overline{R}_{\ve,\eta}:=\{z:\ve\le|z|\le\eta\}.
\end{equation}
Indeed,
$\DF{t_n(z)}{b(z)}\prod\limits_{i=2}^{k_n}\big(z-\alpha^{(n)}_i\big)
=\DF{f_n(z)}{b(z)\big(z-\alpha^{(n)}_1\big)}$, and since $\eta<1$
and by \eqref{11} and \eqref{12}, this term tends uniformly to
$\infty$ in $\overline{R}_{\ve,\eta}$.

Now, for every $j$, $2\le j\le k_n$, we have that
$$\DF{\DF{t_n(z)}{b(z)}\prod\limits_{i=2}^{k_n}\left(z-\alpha^{(n)}_i\right)}
{\DF{t_n(z)}{b(z)}\prod\limits_{i=1,i\ne
j}^{k_n}\left(z-\alpha^{(n)}_i\right)}
=\DF{z-\alpha^{(n)}_j}{z-\alpha^{(n)}_1},$$and by \eqref{12} this
term tends uniformly to $1$ as $n\to\infty$. This means, that for
every $1\le j\le k_n$, and $z\in\overline{R}_{\ve,\eta}$,
$\DF{t_n(z)}{b(z)}\prod\limits_{i=1,i\ne
j}^{k_n}\big(z-\alpha^{(n)}_i\big)$ lies in the same quarter
plane, that is,
\begin{equation}
\label{16}\Pi_{n,z}:=\left\{z:\arg\left[\DF{t_n(z)}{b(z)}\prod\limits_{i=2}^{k_n}\left(z-\alpha^{(n)}_i\right)\right]
-\DF\pi4<\arg z<\arg\left[\DF{t_n(z)}{b(z)}\prod\limits_{i=2}^{k_n}\left(z-\alpha^{(n)}_i\right)\right]
+\DF\pi4\right\},
\end{equation}
for large enough $n$.

Now, if $a$ and $b$ are two complex numbers in the same quarter
plane, then $a+b$ also belongs to that quarter plane and
$|a+b|\ge|a|$, $|b|$. We then conclude by \eqref{16} that for each
$z\in\overline{R}_{\ve,\eta}$, we have for large enough $n$,
$$\left|\DF{t_n(z)}{b(z)}\left[\prod\limits_{i=1}^{k_n}\left(z-\alpha^{(n)}_i\right)
\right]'\right|\ge\left|\DF{t_n(z)}{b(z)}\prod\limits_{i=2}^{k_n}\left(z-\alpha^{(n)}_i\right)
\right|,$$ and by \eqref{15} and \eqref{14},  the Claim  is
proved.\end{proof}

Now,
$\DF{t_n(z)}{b(z)}\bigg[\prod\limits_{i=1}^{k_n}\big(z-\alpha^{(n)}_i\big)
\bigg]'$ has for large enough $n$ exactly $k_n-1$ zeros in
$\Delta(0,\eta)$ (by Theorem Lu). Then for large enough $n$ we
have, for every $z$, $|z|=\eta$,
$$\left|M_n(z)-\DF{t_n(z)}{b(z)}\left[
\prod\limits_{i=1}^{k_n}\left(z-\alpha^{(n)}_i\right)\right]'\right|=|z^{\ell}|<
\left|\DF{t_n(z)}{b(z)}\left[
\prod\limits_{i=1}^{k_n}\left(z-\alpha^{(n)}_i\right)\right]'\right|,$$
and by Rouche's Theorem, we get that $M_n$ has $k_n-1$ zeros in
$\Delta(0,\eta)$, a contradiction. Theorem 1 is
proved.\quad\hfill$\square$

\section{Proof of Theorem 2}
This proof is similar to the proof of Theorem 1 in \cite{4}. By
our Theorem 1, we need only to prove the case that $k\ge3$. By
Theorem CFZ2, $\CF$ is normal at every point $z_0\in D$ at which
$h(z_0)\ne0$ (so that $\CF$ is quasinormal in $D$ ). Consider
$z_0\in D$  such that $h(z_0)=0$. Without loss of generality, we
can assume that $z_0=0$, and then $h(z)=z^{\ell}b(z)$, where
$\ell(\ge1)$ is an integer, $b(z)\ne0$ is an analytic function in
$\Delta(0,\delta)$. We take a subsequence
$\{f_n\}^\infty_{1}\subset\CF$, and we want to prove that
$\{f_n\}$ is not normal at $z=0.$ Suppose by negation that
$\{f_n\}$ is not normal at $z=0.$ Since $\{f_n\}$ is normal in
$\Delta'(0,\delta),$ we can assume (after renumbering) that
$f_n\Rightarrow F $ on $\Delta'(0,\delta)$. If
$F(z)\not\equiv\infty,$ then it is a holomorphic function; hence
by the maximum principle, $F $ extends to be analytic also at
$z=0$, and so $f_n\Rightarrow F$ on $\Delta(0,\delta) $, and we
are done. Hence we assume that
\begin{equation}
\label{17} f_n(z)\Longrightarrow \infty\quad\text{on}\quad
\Delta'(0,\delta).
\end{equation}

Define $\mathcal F_1=\left\{F=\DF{f_n}{h}:n\in\mathbb N\right\}.$
It is enough to prove that $\mathcal F_1$ is normal in
$\Delta(0,\delta).$ Indeed, if (after renumbering)
$\DF{f_n(z)}{h}\Rightarrow H(z) $ on $\Delta(0,\delta),$ then
since $h\ne0$ in $\Delta'(0,\delta)$, it follows  from \eqref{17}
that $H(z)\equiv\infty$ in $\Delta'(0,\delta)$, and thus
$H(z)\equiv\infty$ also in $\Delta(0,\delta).$ In particular,
$\DF{f_n}{h}(z)\ne0$ on each compact subset of $\Delta(0,\delta)$
for large enough $n.$ Since $h\ne0$ on $\Delta'(0,\delta)$ and
since $f_n(0)\ne0$ for every $n\ge1$ by  assumptions of the
theorem, we obtain $f_n(z)\ne0$ on each compact subset of
$\Delta(0,\delta)$ for large enough $n.$ Then by the minimum
principle, it follows from \eqref{17} that $f_n(z)\Rightarrow
\infty$ on $\Delta(0,\delta)$, and this implies the normality of
$\mathcal F.$ So suppose to the contrary that $\CF_1$ is not
normal at $z=0$. By \lemref{lemma1} and the assumptions of Theorem
2, there exist (after renumbering) points $z_n\to0$,
$\rho_n\to0^+$ and a nonconstant meromorphic function on $\C$,
$g(\zeta)$ such that
\begin{equation}\label{18} g_n(\zeta)=\frac{F_n(z_n+\rho_n\zeta)}
{\rho^k_n}=\frac{f_n(z_n+\rho_n\zeta)}{\rho^k_nh(z_n+\rho_n\zeta)}
\overset\chi\Longrightarrow g(\zeta)\quad\text{on}\quad\C,
\end{equation}
all of whose zeros have multiplicity at least $k$ and
\begin{equation}
\label{19}\text{for every}\quad\zeta\in\C,\quad g^\sharp(\zeta)\le
g^\sharp(0)=kA+1,
\end{equation}
where $A>1$ is a constant. Here we have used \lemref{lemma1} with
$\alpha=k$. Observe that $g_n(z)=0$ implies $g^{(k)}_n(\zeta)=1$ and
so $A$ can be  chosen to be any number such that $A\ge1.$ After
renumbering we can assume that $\{z_n/\rho_n\}^\infty_{n=1}$
converges. We separate now into two cases.

\medskip

\noindent\underbar{Case (A)}
\begin{equation}
\label{20}\frac{z_n}{\rho_n}\to\infty.
\end{equation}

\begin{claim} $(1)$ $g(\zeta)=0\Longrightarrow g^{(k)}(\zeta)=1$;\
$(2)$ $g^{(k)}(\zeta)=1\Longrightarrow g^{(k+1)}(\zeta)=0$.
\end{claim}

\begin{proof} Observe that from \eqref{18} and the fact that $h(z)\ne0$ in
$\Delta'(0,\delta),$ it follows that $g$ is an entire function.
Suppose that $g(\zeta_0)=0$. Since $g(\zeta)\not\equiv0$, there
exist $\zeta_n\to\zeta_0$, such that $g_n(\zeta_n)=0$, and thus
$f_n(z_n+\rho_n\zeta_n)=0$. Since $f_n$ and $h$ has no common
zeros, it follows by the assumption that $\zeta_n$ is a zero of
multiplicity $k$ of $g_n(\zeta)$. By Leibniz's rule, and condition
(a) of Theorem 2, it follows that $g^{(k)}_n(\zeta_n)=1$ and thus
$g^{(k)}(\zeta_0)=1$.

For the proof of the other part of the Claim, observe first that by
\eqref{20} we have $$\DF{f_n(z_n+\rho_n\zeta)}{\rho_n^kz^\ell_n}
\Rightarrow g(\zeta)\quad\text{on}\quad\C,$$and thus
$$\DF{f^{(k)}_n(z_n+\rho_n\zeta)}{z^\ell_n}\Rightarrow g^{(k)}(\zeta)\quad\text{on}\quad\C,$$
and then again by \eqref{19} we get that
$$\DF{f^{(k)}_n(z_n+\rho_n\zeta)}{h(z_n+\rho_n\zeta)}\Rightarrow g^{(k)}(\zeta)\quad\text{on}\quad\C.$$
Thus, if there exists $\zeta_0\in\C$, such that
$g^{(k)}(\zeta_0)=1$, there exists a sequence $\zeta_n\to\zeta_0$,
such that $f^{(k)}_n(z_n+\rho_n\zeta_n) =h(z_n+\rho_n\zeta)\ne0$.
By assumption (b) of Theorem 2 we get that
$f^{(k+1)}_n(z_n+\rho_n\zeta_n)=0$, and letting $n$ tend  to
$\infty$  we get that $g^{(k+1)}(\zeta_0)=0$. The Claim is
proved.\end{proof}

We conclude by \lemref{lemma2} and by \lemref{lemma4}  that
$g(\zeta)=\DF{(\zeta-b)^k}{k!}$ for some $b\in\C$ (observe that $g$
is holomorphic by \eqref{20}). By calculation we get that
$$g^\sharp(0)=\DF{|b|^{k-1}/(k-1)!}{1+|b|^{2k}/k!^2}.$$Then if $|b|\le1$, we get
that $g^\sharp(0)\le\DF1{(k-1)!}$, and if $|b|\ge1$, then
$g^\sharp(0)\le\DF k2$. In either case, we get a contradiction to
\eqref{19}.

\medskip

\noindent\underbar{Case (B)}
\begin{equation}
\label{21}\frac{z_n}{\rho_n}\to\alpha\in\C.
\end{equation}
As in Case (A), it follows that $g(\zeta_0)=0\Longrightarrow
g^{(k)}(\zeta_0)=1$. Now set
$$G_n(\zeta)=\frac{f_n(\rho_n\zeta)}{\rho^{k+\ell}_n}.
$$
From \eqref{18} and \eqref{21} we have
\begin{equation}
\label{22}G_n(\zeta)\Longrightarrow
G(\zeta)=g(\zeta-\alpha)\zeta^{\ell}b(0)\quad\text{on}\quad\C.
\end{equation}
Indeed,
$$\frac{f_n(\rho_n\zeta)}{\rho_n^{k+\ell}}=\frac{f_n(\rho_n\zeta)}{\rho^k_nh(\rho_n\zeta)}\cdot
\frac{h(\rho_n\zeta)}{\rho_n^{\ell}}=\frac{f_n\left(z_n+\rho_n\left
(\zeta-\DF{z_n}{\rho_n}\right)\right)}{\rho^k_nh\left(z_n+\rho_n
\left(\zeta-\DF{z_n}{\rho_n}\right)\right)}\frac{(\rho_n\zeta)^{\ell}
b(\rho_n\zeta)}{\rho_n^{\ell}}$$ (cf. \cite[p.~7]{12}). Since $g $
has a pole of order $\ell$ at $\zeta=-\alpha$ (here we use the
fact that for every $n$, $h$ has no common zeros with $f_n$) and
since $\{G_n\}$ are analytic, we have
\begin{equation}
\label{23}G(0)\ne0,\ \infty.
\end{equation}

We now consider several subcases, depending on the nature of $G$.

\noindent\underbar{Case (BI)} \textit{$G $ is a polynomial.}

Since $\{f_n\}$ is not normal at $z=0$, there exist (after
renumbering) a sequence $z^\ast_n\to0$ such that
\begin{equation}
\label{24}f_n(z^\ast_n)=0.
\end{equation}
Otherwise, there is some $\delta'$, $0<\delta'<\delta$ such that
(before renumbering) $f_n(z)\ne0$ in $\Delta(0,\delta')$, and
since $f_n(z)\Rightarrow\infty$ on $\Delta'(0,\delta)$ we would
have by the minimum principle that $f_n(z)\Rightarrow\infty$ on
$\Delta(0,\delta)$, a contradiction to the non-normality of
$\{f_n\}$ at $z=0$. We have that all the zeros of $g$ are of
multiplicity exactly $k$. Then  by \eqref{22} and \eqref{23}, it
follows that all the zeros of $G$ are also of multiplicity exactly
$k$. We consider now two possibilities.

\medskip

\noindent\underbar{Case (BI1)} $\deg(G)=0$.

We can assume that $z^\ast_n$ from \eqref{24} is the closest zero of
$f_n $ to the origin. Then we have
\begin{equation}
\label{25}\frac{f_n(\rho_n\zeta)}{\rho^{k+\ell}_nb(\rho_n\zeta)}\Longrightarrow
\frac{G(0)}{b(0)}\quad\text{on}\quad\C.
\end{equation}
By \eqref{25} we have
\begin{equation}
\label{26}\frac{z^\ast_n}{\rho_n}\to\infty.
\end{equation}
Define $t_n(\zeta)=f_n(z^\ast_n\zeta)/\left(z^{\ast
k+\ell}_nb(z^\ast_n\zeta)\right)$. We want to show that
$\{t_n(\zeta)\}$ is normal in~$\C^\ast$. For this purpose set
$\tilde{t}_n(\zeta)=f_n(z^\ast_n\zeta)/z^{\ast k+\ell}_n$. Since
$b(0)\ne0$, $\infty$ and $z^\ast_n\to0$,   the normality of
$\{t_n\}$ is equivalent to the normality of $\{\tilde{t}_n\}$, and
the latter follows by \lemref{lemma7}. Now, if $\{t_n\}$ is not
normal at $\zeta=0$, then we can write (after renumbering)
$t_n(\zeta)\Rightarrow\infty$ on $\C^\ast$; but $t_n(1)=0$, so
this is not possible. Hence $\{t_n(\zeta)\}$ is normal at
$\zeta=0$. By \eqref{25} and \eqref{26}, $t_n(0)\to0$  as
$n\to\infty$; and thus since $t_n(\zeta)\ne0$ in $\Delta(0,1/2),$
we get by Hurwitz's Theorem that $t_n(\zeta)\Rightarrow0$ on $\C$.
But $t_n(1)=0$; so by assumption (b) of Theorem 2, we get that
$t^{(k)}_n(1)=1$, a contradiction.

\medskip
 \noindent\underbar{Case (BI2)} $G^{(k)}\equiv
b(0)\zeta^\ell$.

Then we have $G^{(k-1)}(\zeta)=\DF{b(0)\zeta^{\ell+1}}{\ell+1}+C$
and
$G^{(k-2)}(\zeta)=\DF{b(0)\zeta^{\ell+2}}{(\ell+1)(\ell+2)}+C\zeta+D$,
where $C$ and $D$ are two constants. Since all zeros of $G$ have
multiplicity exactly $k$, then for any zero $\widehat{\zeta}$ of
$G$, we have
$G^{(k-2)}(\widehat{\zeta})=G^{(k-1)}(\widehat{\zeta})=0$. So
\begin{equation}
\label{27}\DF{\widehat{\zeta}^{\ell+1}}{\ell+1}+C=0,\quad\text{and}\quad
\DF{\widehat{\zeta}^{\ell+2}}{(\ell+1)(\ell+2)}+C\widehat{\zeta}+D=0.
\end{equation}
By calculation, we have $\DF{(\ell+1)C}{\ell+2}\widehat{\zeta}=-D$. If
$CD=0$, then by \eqref{27}, $\widehat{\zeta}=0$, a contradiction. If
$CD\ne0$, then $\widehat{\zeta}=-\DF{(\ell+2)D}{(\ell+1)C}$, which implies
that $G$ has only one zero $\zeta_0$, and then
$$G(\zeta)=\DF{b(0)\zeta_0^{\ell}(\zeta-\zeta_0)^k}{k!}.$$
This contradicts $G^{(k)}\equiv b(0)\zeta^\ell$.

\medskip

\noindent\underbar{Case (BI3)} $G$ is a nonconstant polynomial and
$G^{(k)}\not\equiv b(0)\zeta^\ell$.

Since all zeros of $G$ have multiplicity exactly $k$, we may assume
that$$G=A\prod\limits_{j=1}^t(\zeta-\zeta_j)^k.$$
where $A\ne0$ is a constant and $\zeta_j\ne0$, $j=1,2,\cdots,t$.

\begin{claim}\label{claim4}
$G(\zeta)=0\Longrightarrow G^{(k)}(\zeta)=b(0)\zeta^\ell\Longrightarrow
G^{(k+1)}(\zeta)=0.$
\end{claim}

\begin{proof} Suppose first that $G(\zeta_0)=0$. Then there exists a
sequence, $\zeta_n \to\zeta_0$, such that $f_n(\rho_n\zeta_n)=0$,
and thus $f_n^{(k)}(\rho_n\zeta_n)= (\rho_n\zeta_n)^\ell
b(\rho_n\zeta_n)$, that is,
$\DF{f_n^{(k)}(\rho_n\zeta_n)}{\rho^\ell_n}=
\zeta^\ell_nb(\rho_n\zeta_n)$. In the last equation, the left hand
side tends to $\zeta^\ell_0b(0)$  as $n\to\infty$. This proves the
first part of the Claim.

\medskip

Suppose now that $G^{(k)}(\zeta_0)=b(0)\zeta^\ell_0$. Since
$G^{(k)}(\zeta)\not\equiv b(0)\zeta^\ell$, there exists a sequence
$\zeta_n\to\zeta_0$, such that
$\DF{f_n^{(k)}(\rho_n\zeta_n)}{\rho^\ell_n}=
\zeta^\ell_nb(\rho_n\zeta_n)$, that is, $f_n^{(k)}(\rho_n\zeta_n)=
(\rho_n\zeta_n)^\ell b(\rho_n\zeta_n)$, and thus
$f_n^{(k+1)}(\rho_n\zeta_n)=0$. Since $\DF{f^{(k+1)}_n(\rho_n\zeta
)}{\rho^{\ell-1}_n}\Rightarrow G^{(k+1)}(\zeta)$, we deduce that
$G^{(k+1)}(\zeta_0)=0$, and this completes the proof of the Claim.
\end{proof}

It follows from   Claim \ref{claim4} that $G^{(k+1)}(\zeta_j)=0$,
for $1\le j\le t$.

If $t\ge2$, we know that for every $1\le j\le t$,
 \begin{align*}G^{(k+1)}(\zeta)&=A\left[\prod\limits_{j=1}^t(\zeta-\zeta_j)^k\right]^{(k+1)}\\
&=A\left\{
\sum\limits_{\mu=0}^{k+1}\binom{k+1}{\mu}\left[(\zeta-\zeta_j)^k\right]^{(k+1-\mu)}
\left[\prod\limits_{i=1,i\ne j}^t(\zeta-\zeta_i)^k\right]^{(\mu)}
\right\}\\
&=A\left\{(k+1)k!\left[\prod\limits_{i=1,i\ne
j}^t(\zeta-\zeta_i)^k\right]'+(\zeta-\zeta_j)P_j(\zeta)\right\},
\end{align*}
where $P_j$ is a polynomial. Thus,  by   Claim \ref{claim4} we
have
\begin{equation}
\label{28}\left[\prod\limits_{i=1,i\ne
j}^t(\zeta-\zeta_i)^k\right]'\Bigg|_{\zeta_j}=0,\quad1\le j\le t.
\end{equation}
This means that for every $1\le j\le t$,
$$\sum\limits_{i=1\atop i\ne j}^t(\zeta-\zeta_j)^{k-1}\prod\limits_{\ell=1\atop
\ell\ne i,j}^t(\zeta-\zeta_\ell)^k\Bigg|_{\zeta_j}=0.$$Dividing in $\prod\limits_{\ell\ne j}
(\zeta_j-\zeta_\ell)^{k-1}$ gives
$$\sum\limits_{i=1\atop i\ne j}^t\prod\limits_{\ell=1\atop
\ell\ne i,j}^t(\zeta_j-\zeta_\ell)=0.$$Thus $T''(\zeta_j)=0$ for $1\le j\le t$, where
$T(\zeta)=\prod\limits_{i=1}^t(\zeta-\zeta_i)$.

Now, if $t\ge3$, then $T''$ is of degree $t-2$, and vanishes at
$t$ different points, a contradiction. If $t=2$, we get from
\eqref{28} that
$\left[(\zeta-\zeta_2)^k\right]'\Bigg|_{\zeta_1}=0$ and this is
also a contradiction. So $t=1$ and $G$ has only one zero $\zeta_0\
(\zeta_0\ne0)$, which means that
 $G(\zeta)=\DF{b(0)\zeta_0^{\ell}(\zeta-\zeta_0)^k}{k!}.$

By Hurwitz's Theorem, there exists a sequence
$\zeta_{n,0}\to\zeta_0$, such that $G_n(\zeta_{n,0})=0$. If there
exists $\delta'$, $0<\delta'<\delta$, such that for every $n$
(after renumbering), $f_n(z)$ has only one zero
$z_{n,0}=\rho_n\zeta_{n,0}$ in $\Delta(0,\delta')$.

Set
$$H_n(z)=\frac{f_n(z)}{(z-z_{n,0})^k}.$$
Since $H_n(z)$ is a nonvanishing holomorphic function in
$\Delta(0,\delta')$ and $H_n(z)\Rightarrow\infty$ on
$\Delta'(0,\delta)$, we can deduce as before by the minimum
principle that $H_n(z)\Rightarrow\infty$ on $\Delta(0,\delta')$. But
$$H_n(2z_{n,0})=\frac{f_n(2z_{n,0})}{z^k_{n,0}}=\DF{\rho^\ell_nG_n(2\zeta_{n,0})}{\zeta^k_{n,0}}\to0,$$
a contradiction. Thus, we can assume, after renumbering, that for
every $\delta'>0$, $f_n$ has at least two zeros in
$\Delta(0,\delta')$ for large enough $n$. Thus, there exists
another sequence of points $z_{n,1}=\rho_n\zeta_{n,1}$, tending to
zero, where $z_{n,1}$ is also a zero of $f_n(z)$ and
$\zeta_{n,1}\to\infty$, as $n\to\infty$. We can also assume that
$z_{n,1}$ is the closest zero to the origin of $f_n$, except
$z_{n,0}$.  Now set $c_n=z_{n,0}/z_{n,1}$ and define
$K_n(\zeta)=f_n(z_{n,1}\zeta)/z^{k+\ell}_{n,1}$. By
\lemref{lemma7}, $\{K_n(\zeta)\}$ is normal in $\C^\ast$. Now, if
$\{K_n\}$ is normal at $\zeta=0$, then after renumbering we can
assume that
$$K_n(\zeta)\Longrightarrow K(\zeta)\quad\text{on}\quad\C.$$
If $K(\zeta)\not\equiv$ const., then consider
$$L_n(\zeta):=\frac{K_n(\zeta)}{(\zeta-c_n)^k}.$$
Since $c_n\underset{n\to\infty}\longrightarrow0$, then the
sequence $\{L_n \}^\infty_1$ is normal in $\C^\ast$. It is also
normal at $\zeta=0$. Indeed, $K_n(c_n)=0$ (a zero of order $k$)
and so $L_n $ is a nonvanishing holomorphic function in
$\Delta(0,1)$. Thus (after renumbering)
$$L_n(\zeta)\Longrightarrow\frac{K(\zeta)}{\zeta^k}\quad\text{on}\quad\C.$$
But$$L_n(0)=\frac{K_n(0)}{(-c_n)^k}=\frac{G_n(0)}{\zeta^{\ell}_{n,1}(-\zeta_{n,0})^k}
\underset{n\to\infty}\longrightarrow0,\quad (\text{since}\quad
\zeta_{n,1}\underset{n\to\infty}\longrightarrow \infty ),$$ and
$L_n(\zeta)\ne0$ in $\Delta(0,1/2)$; thus
$K(\zeta)/\zeta^k\equiv0$ in $\C$, a contradiction. If, on the
other hand, $K(\zeta)\equiv$ const., then $K(\zeta)\equiv0$ and
$K^{(k)}(1)=0$. But $K^{(k)}(1)=\lim\limits_{n\to
\infty}K^{(k)}_n(1)=
\lim\limits_{n\to\infty}\DF{f^{(k)}_n(z_{n,1})}{z_{n,1}^\ell}=
\lim\limits_{n\to\infty}\DF{h(z_{n,1})}{z_{n,1}^\ell}=
\lim\limits_{n\to\infty}b(z_{n,1})=b(0)$, a contradiction. Hence
we can deduce that $\{K_n\}$ is not normal at $\zeta=0$, and since
$K_n(\zeta)$ is holomorphic in $\Delta$, then
$$K_n(\zeta)\Longrightarrow \infty\quad\text{on}\quad\C^\ast.$$
But $K_n(1)=0$, a contradiction.

\medskip

 \noindent\underbar{Case (BII)} \textit{$G(\zeta)$ is a
transcendental entire function.}

Consider the family
$$\CF(G)=\left\{t_n(z):=\frac{G(2^nz)}{2^{n(k+\ell)}}:n\in\mathbb N\right\}.$$
By  Claim \ref{claim4}, we deduce
\begin{enumerate} \item[(i)] $t_n(z)=0\Longrightarrow t^{(k)}_n(z)=z^\ell$; and
\item[(ii)] $t^{(k)}_n(z)=z^\ell\Longrightarrow t^{(k+1)}_n(z)=0$.
\end{enumerate}
We then get by Theorem CFZ2 that $\CF(G)$ is normal in $\C^\ast$.
Thus there exists $M>0$  such that for every  $z\in
R_{1,2}:=\{z:1\le|z|\le2\}$
$$t^\#_n(z)=\frac{2^{n(k+\ell+1)}|G'(2^nz)|}{2^{2n(k+\ell)}+|G(2^nz)|^2}\le M.$$
Set $r(\zeta):=G(\zeta)/\zeta^{k+\ell}$. Then $r$ is a transcendental
meromorphic function, whose only pole is $\zeta=0$. For every
$\zeta$, $|\zeta|\ge2$ there exists $n\ge1$ and $z\in R_{1,2}$, such
that
\begin{equation}
\label{29}\zeta=2^nz.
\end{equation}
Calculation gives
$$r^\sharp(\zeta)=\frac{|G'(\zeta)\zeta^{k+\ell}-(k+\ell)\zeta^{k+\ell-1}G(\zeta)|}{|\zeta|^{2(k+\ell)}+|G(\zeta)|^2}.$$
Thus, if $|\zeta|\ge2$ satisfies \eqref{29} then
\begin{equation}
\begin{aligned}\label{30}
|\zeta
r^\sharp(\zeta)|&=|2^nz|\frac{|G'(2^nz)(2^nz)^{k+\ell}-(k+\ell)(2^nz)^{k+\ell-1}G(2^nz)|}{|2^nz|^{2(k+\ell)}+|G(2^nz)|
^2}\\
\\
&\le\frac{2^{k+\ell+1}\cdot2^{n(k+\ell+1)}|G'(2^nz)|}{2^{2n(k+\ell)}+|G(2^nz)|^2}+
\frac{(k+\ell)2^{(n+1)(k+\ell)}|G(2^nz)|}{2^{2n(k+\ell)}+|G(2^nz)|^2}.
\end{aligned}
\end{equation}
By separating into two cases, depending on
$|G(2^nz)|>2^{(n+1)(k+\ell)}$ or $|G(2^nz)|\le2^{(n+1)(k+\ell)}$, we
see that the last expression in \eqref{30} is less or equal to
$$2^{k+\ell+1}t^\sharp_n(z)+(k+\ell)2^{2(k+\ell)}.$$Thus, to every $|\zeta|\ge2$,
$$|\zeta r^\sharp(\zeta)|\le M\cdot2^{k+\ell+1}+(k+\ell)2^{2(k+\ell)}.$$
But, according to Theorem B,
$\varlimsup\limits_{\zeta\to\infty}|\zeta|r^\sharp(\zeta)=\infty$,
and we thus have a contradiction (cf. \cite[pp.~19-21]{3}).
Theorem 2 is proved.\quad\hfill$\square$
\section{Proof of Theorem 3}
By Theorem CFZ3, $\CF$ is normal at every point $z_0\in D$ at
which $h(z_0)\ne0$ (so that $\CF$ is quasinormal in $D$). Consider
$z_0\in D$  such that $h(z_0)=0$. Without loss of generality, we
can assume that $z_0=0$, and then $h(z)=z^{\ell}b(z)$, where
$\ell(\ge1)$ is an integer, $b(z)\ne0$ is an analytic function in
$\Delta(0,\delta)$. We take a subsequence
$\{f_n\}^\infty_{1}\subset\CF$, and we only need to prove that
$\{f_n\}$ is not normal at $z=0.$

Define $\mathcal F_2=\left\{F=\DF{f_n}{h}:n\in\mathbb N\right\}.$ It
is enough to prove that $\mathcal F_2$ is normal in
$\Delta(0,\delta).$ Suppose to the contrary that $\CF_2$ is not
normal at $z=0$. By \lemref{lemma1} and the assumptions of Theorem
3, there exist (after renumbering) points $z_n\to0$, $\rho_n\to0^+$
and a nonconstant meromorphic function on $\C$, $g(\zeta)$ such that
\begin{equation}\label{31} g_n(\zeta)=\frac{F_n(z_n+\rho_n\zeta)}
{\rho^2_n}=\frac{f_n(z_n+\rho_n\zeta)}{\rho^2_nh(z_n+\rho_n\zeta)}
\overset\chi\Longrightarrow g(\zeta)\quad\text{on}\quad\C,
\end{equation}
all of whose zeros are multiple and
\begin{equation}
\label{32}\text{for every}\quad\zeta\in\C,\quad g^\sharp(\zeta)\le
g^\sharp(0)=2A+1,
\end{equation}
where $A>1$ is a constant. After renumbering we can assume that
$\{z_n/\rho_n\}^\infty_{n=1}$ converges. We separate now into two
cases.

\medskip

\noindent\underbar{Case (A)}\quad
$\DF{z_n}{\rho_n}\to\infty$.

Similar to the proof of Theorem 2, we can prove that
$g(\zeta)=0\Longrightarrow g''(\zeta)=1$ and that
$g''(\zeta)=1\Longrightarrow g'''(\zeta)=g^{(s)}(\zeta)=0$. Then
by Lemmas \ref{lemma4} and \ref{lemma3}, we
have$$g(\zeta)=\DF{(\zeta-b)^2}2,$$ for some $b\in\C$. Thus
$g^\sharp(0)=\DF{|b|}{1+|b|^4/4}$ and then $g^\sharp(0)\le1$,
which contradicts \eqref{32}.

\medskip

\noindent\underbar{Case (B)}
\begin{equation}
\label{33}\frac{z_n}{\rho_n}\to\alpha\in\C.
\end{equation}
As in the proof of Theorem 2, we have $g(\zeta_0)=0\Longrightarrow
g''(\zeta_0)=1$. Now set
$G_n(\zeta)=\DF{f_n(\rho_n\zeta)}{\rho^{2+\ell}_n}.$ From
\eqref{31} and \eqref{33} we have
$$G_n(\zeta)\Longrightarrow G(\zeta)=b(0)g(\zeta-\alpha)\zeta^{\ell}\quad\text{on}\quad\C.$$
Since $g $ has a pole of order $\ell$ at $\zeta=-\alpha$,
 $G(0)\ne0,\ \infty.$

We now consider several subcases, depending on the nature of $G$.

\noindent\underbar{Case (BI)} \textit{$G $ is a polynomial.}

By a similar method of   proof used in the proof of Theorem 2 (and
using \lemref{lemma8} instead of  \lemref{lemma7} in the
appropriate places), we can get
$$G(\zeta)=\DF{b(0)\zeta_0^{\ell}(\zeta-\zeta_0)^2}{2},$$
and also we can arrive at a contradiction.

\noindent\underbar{Case (BII)} \textit{$G(\zeta)$ is a
transcendental entire function.}

Consider the family
$$\CF(G)=\left\{t_n(z):=\frac{G(2^nz)}{2^{n(2+\ell)}}:n\in\mathbb N\right\}.$$
We have
\begin{enumerate} \item[(i)] $t_n(z)=0\Longrightarrow t''_n(z)=z^\ell$; and
\item[(ii)] $t''_n(z)=z^\ell\Longrightarrow t'''_n(z)=t^{(s)}_n(z)=0$.
\end{enumerate}
We then get by Theorem CFZ3 that $\CF(G)$ is normal in $\C^\ast$.
Set $r(\zeta):=G(\zeta)/\zeta^{2+\ell}$, and we have that, for
every $\zeta$, $|\zeta|\ge2,$ there exists $n\ge1$ and $z\in
R_{1,2}$, such that $$|\zeta r^\sharp(\zeta)|\le
M\cdot2^{2+\ell+1}+(2+\ell)2^{2(2+\ell)}.$$ But, according to
Theorem B,
$\varlimsup\limits_{\zeta\to\infty}|\zeta|r^\sharp(\zeta)=\infty$,
and we thus have a contradiction (cf. \cite[pp.~19-21]{3}).
Theorem 3 is proved. \hfill\quad$\square$

\bibliographystyle{amsplain}

\end{document}